%% file: main.tex
\documentclass[journal,twoside,web]{ieeecolor}
\usepackage{lcsys}

\input{settings.tex}

\usepackage{tikz}
\newcommand\copyrighttext{%
  \footnotesize \textcopyright 2024 IEEE. Personal use of this material is permitted.
  Permission from IEEE must be obtained for all other uses, in any current or future
  media, including reprinting/republishing this material for advertising or promotional
  purposes, creating new collective works, for resale or redistribution to servers or
  lists, or reuse of any copyrighted component of this work in other works.
  DOI: \href{https://ieeexplore.ieee.org/document/10518123}{10.1109/LCSYS.2024.3396611}}
  
\newcommand\copyrightnotice{%
\begin{tikzpicture}[remember picture,overlay]
\node[anchor=south,yshift=5pt] at (current page.south) {\fbox{\parbox{\dimexpr\textwidth-\fboxsep-\fboxrule\relax}{\copyrighttext}}};
\end{tikzpicture}%
}

\title{\LARGE \bf
	Fourth-order suboptimality of nominal model predictive control in the presence of uncertainty
}

\author{
	Florian Messerer$^1$, Katrin Baumg\"artner$^1$, Sergio Lucia$^2$, Moritz Diehl$^{1,3}$
	\thanks{
		$^1$Department of Microsystems Engineering (IMTEK), University of Freiburg, 79110 Freiburg, Germany
		(\{first.last\}@imtek.uni-freiburg.de)
		}
	\thanks{
		$^2$Chair of Process Automation Systems, TU Dortmund University, 44227 Dortmund, Germany ({sergio.lucia@tu-dortmund.de})
	}
	\thanks{
		$^3$Department of Mathematics, University of Freiburg, 79104 Freiburg, Germany
			}
	\thanks{
		This research was supported by DFG via projects 423857295, 424107692, 525018088 and Research Unit FOR 2401, by BMWK via 03EI4057A and 03EN3054B, and by the EU via ELO-X 953348.
		}
	}
\begin{document}

\maketitle
\thispagestyle{empty}
\pagestyle{empty}
\copyrightnotice

\begin{abstract}           
	We investigate the suboptimality resulting from the application of nominal model predictive control (MPC) to a nonlinear discrete time stochastic system.
	The suboptimality is defined with respect to the corresponding stochastic optimal control problem (OCP) that minimizes the expected cost of the closed loop system.
	In this context, nominal MPC corresponds to a form of certainty-equivalent control (CEC).
	We prove that, in a smooth and unconstrained setting, the suboptimality growth is of fourth order with respect to the level of uncertainty, a parameter which we can think of as a standard deviation.
	This implies that the suboptimality does not grow very quickly as the level of uncertainty is increased, providing further insight into the practical success of nominal MPC.
	Similarly, the difference between the optimal and suboptimal control inputs is of second order.
	We illustrate the result on a simple numerical example, which we also use to show how the proven relationship may cease to hold in the presence of state constraints.
\end{abstract}

\begin{IEEEkeywords}
	Optimization,
	predictive control for nonlinear systems,
	stochastic optimal control.
\end{IEEEkeywords}

%===============================================================================

\section{Introduction}
\IEEEPARstart{M}{odel} predictive control (MPC)  is a control scheme that, given the current system state, computes control inputs by optimizing the predicted system trajectory \cite{Rawlings2017}. Usually this takes the form of solving online an optimal control problem (OCP).
In its most common form, known as nominal MPC, the predictions are treated as certain even though usually they are associated with significant uncertainty.
In contrast, stochastic MPC explicitly considers uncertainty and optimizes over probability distributions of trajectories.
With respect to the corresponding stochastic OCP, nominal MPC can be considered as a form of suboptimal control \cite{Bertsekas2005a}.
Nonetheless, nominal MPC often yields powerful controllers in practice \cite{Qin2003}, and there is a wide range of theoretical results that support this observation, as outlined in the following.

% \subsection{Related work}

A fundamental theorem from linear control theory is the certainty-equivalence principle \cite{Anderson1990}, which holds for linear systems with quadratic cost and independent noise.
%, i.e., for the linear quadratic regulator (LQR).
In this special case, the control policy resulting from the nominal problem is optimal also for the stochastic problem.
In line with this result, applying the nominally optimal policy to a stochastic system is often referred to as certainty-equivalent control (CEC), also for nonlinear systems.
This is different from simply applying a nominally optimal sequence of fixed control inputs, because the policy reacts to disturbances.
However, in the general nonlinear case, CEC is suboptimal.

Results from control theory are often not explicitly concerned with suboptimality as defined above, but with a similar motivation investigate the stability of nominal MPC under the presence of perturbations \cite{Scokaert1997,Limon2009, McAllister2022a}, referred to as inherent robustness.
This covers also suboptimal solutions of the nominal OCP \cite{Pannocchia2011, Allan2017}, and limitations present in the nonlinear setting \cite{Grimm2004}.

Taking a step back, the solution to a discrete time stochastic OCP can be expressed via dynamic programming (DP) \cite{Bellman1957}, but the computation of the solution is in general intractable.
Arguably, whole fields of study are dedicated to a large extent to finding tractable approximations and analyzing their consequences.
Besides MPC, this includes approximate DP \cite{Bertsekas2005}, which shares major results also with reinforcement learning (RL) \cite{Sutton2020}.%, e.g., in the form of algorithms such as Q-learning.

An overview of the suboptimality resulting from various approximations in the context of DP is given in \cite{Bertsekas2005a}, including bounds on the suboptimality of CEC.
Suboptimality resulting from sampling of the probability distribution is treated in the stochastic programming literature in the context of the sample average approximation \cite{Shapiro2009}.
Other results cover the suboptimality resulting from the optimization over  parametrized policies \cite{Hadjiyiannis2011a}, and -- in a nominal setting -- bounds on the performance loss from finite horizon approximations \cite{Gruene2008,Li2023} of the infinite horizon problem as well as the
transient behavior of the suboptimality \cite{Karapetyan2023a}.

\subsection{Contribution and outline}
In this paper we investigate the dependence of the suboptimality of CEC on the level of uncertainty $\sigma$, a parameter that can be thought of as akin to a standard deviation of the process noise.
We prove that for smooth and unconstrained problems with finite horizon the suboptimality is of size $\bigO(\sigma^4)$.
Similarly, the difference of the control inputs %generated by the corresponding policies
is of size $\bigO(\sigma^2)$.
Under a different set of assumptions and a different line of proof, similar results have been obtained in \cite{Fleming1971} in a continuous time setting and in \cite{Mohamed2022} in an MPC context.

The paper is structured as follows.
In Section~\ref{sec:stochOCP} we define the stochastic OCP and discuss its suboptimal solution via CEC.
We then analyze the resulting suboptimality in Section~\ref{sec:horz_fin}.
In Section~\ref{sec:num} we illustrate the result with a numerical example, followed by a concluding Section~\ref{sec:conc}.

\subsection{Notation and preliminaries}

For two column vectors $x\in\R^n$, $y\in\R^m$, we denote their vertical concatenation by
$(x,y) \defeq [x^\top, y^\top]^\top$, i.e., $(x,y)\in\R^{n+m}$.
The identity matrix is denoted by $I$, with the dimension inferred from context.
The partial derivative with respect to a variable $x$ is denoted by $\dpartialtfrac{}{x}$ and means the derivative with respect to the explicit argument of a function.
The argument can be indicated by parentheses or subscript, e.g., 
$f_\sigma(x)$ has two arguments: $x$ and $\sigma$.
Gradients of functions are denoted by $\nabla_x f(x)\defeq \dpartialtfrac{}{x}f(x)^\top$.
Total derivatives also take into account dependencies of function arguments, and are denoted by $\dtotaltfrac{}{x}$.
For a multivariate scalar-valued function $c\colon\R^n\to \R$, $x\mapsto c(x)$ we denote by $\dpartial{^i}{x^i} c(x)\bullet x^i$, for $i=1,\dots,$ the $i$-th order tensor product resulting in a scalar, consistent with $\dpartial{^1}{x^1} c(x)\bullet x^1 = \nabla c(x)^\top x$ and $\dpartial{^2}{x^2} c(x)\bullet x^2 = x^\top \nabla^2 c(x) x$.

\section{Stochastic optimal control and suboptimality } \label{sec:stochOCP}
Consider the discrete time stochastic system
\begin{equation} \label{eq:stoch_sys}
	x^+ = f(x, u, w),
\end{equation}
with continuous state $x \in\R^{n_x}$, control $u\in \R^{n_u}$, and process noise $w \in \R^{n_w}$.
At each time step the process noise is independently distributed as $w\sim\calW$, with zero mean and unit variance.
% Without loss of generality, we will assume it has zero mean, $\EXV_{w\sim \calW} \{w\} = 0$, and unit variance $\EXV_{w\sim \calW} \{ww^\top\} = \eye$.
%  and unit variance, $\EXV_{w\sim \calW} \{ww^\top\} = \eye$.
% Further, we assume that the distribution function of $\calW$ is symmetric with respect to the mean.
% Note that this is not very restrictive, since the dynamics function $f$ can include nonlinear transformations of the noise.

\subsection{The stochastic optimal control problem}
We aim to optimize the system trajectory over the horizon $N$, with discrete time index $k=0,\dots,N$.
Denoting by
\begin{equation}
\begin{aligned}
		\tilde x_0(\utraj, \wtraj; x) &= x,\\
		\tilde x_{k+1}(\utraj, \wtraj; x) &= f(  \tilde x_{k}(\utraj, \wtraj; x), u_k, w_k),
\end{aligned}
\end{equation}
$k=0,\dots,N-1$, the forward simulation of the system from initial state $x$ under the control and noise trajectories $\utraj=(u_0, \dots, u_{N-1})$ resp. $\wtraj=(w_0,\dots, w_{N-1})$, the total cost incurred by this this trajectory is
\begin{equation}
\begin{multlined}
		J(\utraj, \wtraj; x) \defeq  \sum_{k=0}^{N-1} L(\tilde x_k(\utraj, \wtraj), u_k) +  E(\tilde x_k(\utraj, \wtraj; x)),
\end{multlined}
\end{equation}
with stage cost $L$ and terminal cost $E$.

We want to minimize the expected value of this cost for the \emph{closed loop} system, where each control $u_k$ is chosen only after the corresponding state $x_k$ is known, i.e., after the preceding disturbances have already been realized. 
This corresponds to the recursive optimization problem
\begin{equation} \label{eq:stochOCP}
\begin{multlined}
		V^\star_\sigma(x) =
		\min_{u_0} \EXV_{w_0} \min_{u_1} \EXV_{w_1} \dots \min_{u_{N-1}} \EXV_{w_{N-1}} J(\utraj, \sigma \wtraj;  x).
\end{multlined}
\end{equation}

Here, we introduced the parameter $\sigma \in\R$ which provides us with a convenient way of scaling the influence of the noise.
We refer to $\sigma$ as the level of uncertainty.
We can think about $\sigma$ as akin to a standard deviation, since the effectively applied noise $\sigma w$ has variance $\sigma^2 \eye$. However, we also allow negative values of $\sigma$.

% \begin{remark}
% 	While \eqref{eq:stochOCP} could also be formalized as a Markov decision process (MDP), we will stick to the language and notation of stochastic optimal control.
% \end{remark}

% \subsubsection{Suboptimal control}

\begin{remark}
	For clarity of presentation and for a more lightweight notation we consider \eqref{eq:stochOCP} only for the case of time-invariant dynamics and stage cost.
	However, the results straightforwardly extend to the time-variant setting.
\end{remark}

\subsection{Certainty-equivalent control and nominal MPC}

% Either solve DP recursion for nominal problem, or go implicit by solving NLP for current state.

In this paper we are concerned with the suboptimality resulting from nominal MPC resp. CEC.
By setting $\sigma=0$ in \eqref{eq:stochOCP}, the resulting nominal OCP can be written as the nonlinear program (NLP)
\begin{mini!}
	{\xtraj, \utraj}
	{ \sum_{k=0}^{N-1}  L(x_k, u_k) +  E(x_N) }
	{\label{ocp:nominal}}
	{}
	\addConstraint{x_0}{= x}
	\addConstraint{x_{k+1}}{ = f(x_k, u_k, 0), \quad k=0,\dots,N-1}.
\end{mini!}
% For solving NLP, a wide range of efficient numerical optimization methods are available \cite{Nocedal2006, Rawlings2017}.
The idea of nominal MPC is to solve \eqref{ocp:nominal} for the current state $x$ and to apply the first element of the resulting optimal control vector, $u^\star_0$.
After observing the resulting state, the OCP \eqref{ocp:nominal} is again solved with this new initial state, either for a receding or a shrinking horizon.
In order to isolate the suboptimality resulting from solving the nominal problem from the suboptimality resulting from approximations of the horizon, we will assume a shrinking horizon. 
Since we solve \eqref{ocp:nominal} instead of \eqref{eq:stochOCP}, this leads to suboptimality, even if the control input is recomputed at every time step.

\subsection{Suboptimality of CEC}

In the following, we define the suboptimality of CEC with respect to \eqref{eq:stochOCP}.
As the recursive structure implies, we can conceptually solve \eqref{eq:stochOCP} via DP.
The optimal state value function $V^\star_{\sigma,k}$ and state-action value function $Q^\star_{\sigma,k}$ at time $k$ are recursively defined by
\begin{subequations} \label{eq:optval_finhorz}
	\begin{align}
		V^\star_{\sigma,N}(x) &= E(x),\\
		Q^\star_{\sigma,k}(x,u) &= L(x, u) +  \EXV_w\{V^\star_{\sigma,k+1}(f(x,u,\sigma w))\},\\
		V^\star_{\sigma,k}(x) &= \min_u  Q^\star_{\sigma,k}(x,u), 
		\quad k=N-1,\dots, 0, \label{eq:optval_finhorz_state}
	\end{align}
% $k=N-1,\dots, 0$.
with associated optimal policy
\begin{equation}
	\pi^\star_{\sigma,k}(x) = \argmin_u Q^\star_{\sigma,k}(x,u), \quad k=0,\dots,N-1.
\end{equation}
\end{subequations}

Similarly, the value functions resulting from the evaluation of a given policy $\poltraj = (\pi_0, \dots, \pi_{N-1})$ are defined by
\begin{subequations} \label{eq:poleval_finhorz}
	\begin{align}
		V^\pi_{\sigma,N}(x) &= E(x),\\
		Q^\pi_{\sigma,k}(x,u) &= L(x, u) +  \EXV_w\{V^\pi_{\sigma,k+1}(f(x,u,\sigma w))\},\\
		V^\pi_{\sigma,k}(x) &= Q^\pi_{\sigma,k}(x,\pi_k(x)), \quad k=N-1,\dots, 0.  \label{eq:poleval_finhorz_state}
	\end{align}
\end{subequations}

% We point out that evaluating the optimal policy $\pi_\sigma^\star$ yields the optimal value function $V^\star_\sigma$,
% \begin{equation}\label{eq:opt_val_is_eval_of_opt_pol}
% 	V^{\pi^\star}_\sigma(x) = Q_\sigma^{\pi^\star}(x, \pi^\star(x))
% 	= Q^\star_\sigma(x, \pi^\star_\sigma(x)) = V^\star_\sigma(x).
% \end{equation}

While \eqref{eq:optval_finhorz} defines the solution to \eqref{eq:stochOCP}, for $\sigma\neq 0$ this will in general be intractable without approximations.
However, by setting $\sigma=0$ in \eqref{eq:optval_finhorz}, we obtain the DP recursion corresponding to the nominal OCP \eqref{ocp:nominal}.
% In this case,  \eqref{ocp:nominal} allows us to define the optimal value functions and policies nonrecursively via nonlinear programs.
% In particular, the nominally optimal policy for $k=0$, $\pi^\star_{\sigma,0}(x)\vert_{\sigma=0}$, corresponds to the first element $u_0^\star$ of the optimal control trajectory defined by \eqref{ocp:nominal} given initial state $x$.
% Similarly, $\pi_{\sigma,k}(x)\vert_{\sigma=0}$, $k=0,\dots,N-1$, is the first control of the solution to OCP \eqref{ocp:nominal} defined over a shortened horizon of $N-k$.
The idea of CEC is to apply the policy obtained by solving \eqref{eq:optval_finhorz} for $\sigma=0$, even if the system actually follows dynamics with $\sigma \neq 0$.
We introduce the shorthands
\begin{subequations} \label{eq:poleval_finhorz_cec}
	\begin{alignat}{6}
		\pi^\ce_k(x) &\defeq \pi^\star_{0,k}(x), &&\quad k=0,\dots,N-1,\\
		Q^\ce_{\sigma,k}(x,u) &\defeq Q^{\pi^\ce}_{\sigma, k}(x,u), &&\quad k=0,\dots,N-1,\\
		V^\ce_{\sigma,k}(x) &\defeq V^{\pi^\ce}_{\sigma, k}(x),  &&\quad k=0,\dots,N,
	\end{alignat}
\end{subequations}
for this policy and its evaluation \eqref{eq:poleval_finhorz} on a system with uncertainty level $\sigma$.

The resulting suboptimality is defined as the difference of the optimal value function at $k=0$ and the value function resulting from the evaluation of $\poltraj^\ce$,
\begin{equation} \label{eq:subopt_finhorz}
	\Delta V_\sigma(x) \defeq V_{\sigma}^\ce(x) - V^\star_{\sigma}(x),
\end{equation}
where we dropped the time index, $V_\sigma^\ce(x) \defeq V_{\sigma,0}^\ce(x)$, and $V^\star_\sigma(x) = V^\star_{\sigma,0}(x)$ as in \eqref{eq:stochOCP}.
If the real system has $\sigma=0$, the suboptimality of CEC is trivially zero, and, intuitively, it will grow as $\sigma$ is increased.
Note that here, in order to isolate the effect due to CEC, we do not consider the effect of additional approximations, e.g., of the horizon.
An important limitation in practice is that in DP minimization is commonly understood in a global sense, whereas numerical methods for solving \eqref{ocp:nominal} can in general only guarantee local optimality \cite{Rawlings2017}.
However, since our focus is on the suboptimality of CEC, we do not address this further.

\section{Analysis of suboptimality}\label{sec:horz_fin}
We will now characterize in more detail how the suboptimality \eqref{eq:subopt_finhorz} depends on $\sigma$.
For the derivation of the results, it will be useful to consider the DP operator $T_\sigma$ associated with the recursion \eqref{eq:optval_finhorz}.
Given a value function $V\colon\R^{n_x} \to \R$, this operator defines the updated value function as
\begin{equation} \label{eq:DPoperator}
	T_\sigma[V](x) \defeq \min_u L(x,u) +  \EXV_w\{ V(f(x,u,\sigma w))\}.
\end{equation}
This allows us to write the optimal value function at stage $k$ as the $(N-k)$-fold composition of the DP operator applied to the terminal cost,
$
V_{\sigma,k} = (T_\sigma)^{N-k} [E].
$

It will also be useful to decompose \eqref{eq:DPoperator} into suboperations,
% \begin{subequations} \label{eq:DPoperator_sub}
		\begin{align}
			T_\sigma^{V\to Q}[V](x,u) &\defeq  L(x,u) +  \EXV_w \{ V(f(x,u,\sigma w))\}, \\
			T_\sigma^{V\to \pi}[V](x) &\defeq\argmin_u\;T_\sigma^{V\to Q}[V](x,u),\\
			T_\sigma^{Q\times\pi \to V}[Q,\pi](x) & \defeq Q(x, \pi(x)), \label{eq:DPop_qpi_to_V}
		\end{align}
% \end{subequations}
such that
$
	T_\sigma[V] = T_\sigma^{Q\times \pi \to V}[T_\sigma^{V\to Q}[V],T_\sigma^{V\to \pi}[V]].
$

Similarly, we define the DP policy evaluation operator as
\begin{align} \label{eq:DPoperator_pol}
	\tilde T_\sigma[V, \pi](x)
	\defeq%&\; L(x, \pi(x)) +  \EXV_w \{ V_\sigma(f(x,\pi(x),\sigma w))\}
	%\nonumber
	%\\
	% =&\;
	T_\sigma^{Q\times \pi \to V}[T_\sigma^{V\to Q}[V],\pi](x).
\end{align}
% For policies, we make the following notational distinction:
% Arbitrary policies, possibly depending on $\sigma$, are denoted by $\pi_\sigma$.
% Greedy policies, resulting from the application of the DP operator to an \emph{arbitrary} value function $V_\sigma(x)$, are denoted as $\polgreedy_\sigma = T_\sigma^{V\to \pi}[V_\sigma]$.
% In contrast, we denote policies that are optimal with respect to \eqref{eq:stochOCP} resp. \eqref{eq:optval_finhorz} by $\pi^\star_{\sigma,k} = T_\sigma^{V\to \pi}[V^\star_{\sigma,k+1}]$, $k=0,\dots,N-1$, i.e., they are greedy policies resulting from \emph{optimal} value functions.   

Our main result will be based on a Taylor expansion of the suboptimality \eqref{eq:subopt_finhorz} with respect to $\sigma$.
For this purpose, we first establish several lemmata on how the DP operators defined above preserve derivative related properties.
The assumptions we need for this are mostly technical and ensure that all necessary derivatives with respect to $\sigma$ exist.
The zero mean and unit variance assumption on the noise will simplify some arguments, but is without loss of generality, as we can always correspondingly adapt the definition of the dynamics $f$ via incorporation of affine transformations.

	\begin{assumption}[Smoothness]\label{ass:smoothfuncs}
	The dynamics function $f$, the stage cost function $L$ and the terminal cost $E$ are smooth with respect to all arguments, $f, L, E\in C^\infty$.
	% More specifically, we require them to be four times differentiable.
	\end{assumption}
	\begin{assumption}[Noise distribution]\label{ass:noise}
	At each time point, the noise $w$ independently follows the probability distribution~$\calW$.
	This distribution has zero mean, $\EXV_{w\sim \calW} \{w\} = 0$,
	unit variance, $\EXV_{w\sim \calW} \{ww^\top\} = \eye$,
	and is supported only on a compact set $W\subset\R^{n_w}$.
	\end{assumption}
	\begin{remark}[Support of the noise distribution] \label{rem:finite_tail}
		We introduce the bounded noise support to ensure that all expectations are finite and that we can compute their derivatives.
		% This follows because -- in combination with Assumption~\ref{ass:smoothfuncs} -- all expectations correspond to Lebesgue integrals of continuous functions over bounded intervals.
		In principle, the results could also be derived for distributions with sufficiently fast decaying tails, e.g., normal distributions.
		This would require some additional assumptions on the considered functions to ensure they do not counteract the tail decay.
		In more detail, they need to be Lebesgue-integrable with respect to the measure space corresponding to $\calW$, cf., e.g., \cite[Thm. 2.27]{Folland1999}.
		% For most practical purposes, we can ensure bounded support by replacing normal distributions with modified normal distributions that are truncated far off in the tails, such that the disregarded probability mass is negligible.
	\end{remark}
	\begin{assumption}[Regularity]\label{ass:regularity_finhorz}
		For the considered value functions $V_\sigma$,
		the DP operator \eqref{eq:DPoperator} is associated with a
		unique minimizer $\pi_{\sigma}(x) = T_\sigma^{V\to\pi}[V_\sigma](x)$ for all $x$ and $\sigma \in [-\bar\sigma,\bar\sigma]$, for some $\bar\sigma\in\R_{++}$.
		At this solution, the second order sufficient condition holds, i.e., $\nabla_u^2 Q_{\sigma}(x, \pi_{\sigma}(x)) \succ 0$, where $Q_\sigma = T_\sigma^{V\to Q}[V_\sigma]$ is the associated state-action value function.
		More specifically, we assume this to hold at every step of the DP recursion \eqref{eq:optval_finhorz}.
	\end{assumption}

	\begin{lemma}\label{lem:DPeval_op_smooth}
		Let Assumptions~\ref{ass:smoothfuncs} to \ref{ass:regularity_finhorz} hold.
		The operators $T_\sigma^{V\to Q}$, $T_\sigma^{V\to\pi}$, $T_\sigma^{Q\times \pi \to V}$, $T_\sigma$,   $\tilde T_\sigma$  as defined in \eqref{eq:DPoperator} to \eqref{eq:DPoperator_pol} preserve the smoothness of the respective functions with respect to all arguments.
	\end{lemma}
	\begin{proof}
		Let $\bar V_\sigma$ be a smooth value function and denote $Q_\sigma = T_\sigma^{V\to Q}[\bar V_\sigma]$, 
		$\pi_\sigma = T_\sigma^{V\to\pi}[\bar V_\sigma]$, $V_\sigma = T_\sigma[\bar V_\sigma] = \tilde T_\sigma[\bar V_\sigma, \pi_\sigma]$.
		Then
		\begin{subequations}
			\begin{align}
				Q_\sigma(x, u) &= L(x,u) + \EXV_w \{ \bar V_\sigma(f(x,u,\sigma w))\}, \\
				\label{eq:pol_greedy}
				\pi_\sigma(x) &=  \argmin_u Q_\sigma(x,u),\\
				V_\sigma(x) &= Q_\sigma(x,\pi_\sigma(x)). \label{eq:valfunc_smoothnesslemma}				
			\end{align}
		\end{subequations}
		Smoothness of $Q_\sigma$ follows from smoothness of $\bar V_\sigma $, $L$, $f$, (Assumption~\ref{ass:smoothfuncs}) and the measure theoretic statement of the Leibniz integral rule \cite[Thm. 2.27]{Folland1999}, which can be applied since the bounded noise support (Assumption~\ref{ass:noise}) in combination with smoothness ensures boundedness of both the expectation and the expectation of the derivatives.
		Thus, $T_\sigma^{V\to Q}$ preserves smoothness.
		Due to \eqref{eq:pol_greedy}, the policy $\pi_\sigma$ is implicitly defined via
		$
		% \begin{equation} \label{eq:pol_greedy_impl}
			\nabla_u Q_\sigma(x, \pi_\sigma(x)) = 0.
		% \end{equation}
		$
		Then smoothness of $\pi_\sigma$ follows from  the implicit function theorem, which can be applied due to smoothness of $Q_\sigma$ and the regularity assumption (Assumption~\ref{ass:regularity_finhorz}).
		Thus, $T_\sigma^{V\to\pi}$ preserves smoothness.
		Further, $T_\sigma^{Q\times \pi \to V}$ trivially preserves smoothness, cf. \eqref{eq:valfunc_smoothnesslemma}.
		Since $T_\sigma$ and $\tilde T_\sigma$ are compositions of the preceding three operators, the property transfers.  
	\end{proof}

	Having established the technical necessities, we will now derive the main result.
	For this purpose we first establish several consequences of applying the DP operator $T_\sigma$ to a value function $\bar V_\sigma$ of which the first-order derivative with respect to $\sigma$ is zero at zero, i.e., $\dpartialtfrac{}{\sigma}\bar V_\sigma(x)\vert_{\sigma=0}$ for all $x$.
	This is trivially true for the terminal cost $E$ since it does not depend on $\sigma$. In Lemma~\ref{lem:DPoperator_preservation} we show that this property is preserved by the DP operator.
	In Lemma~\ref{lem:DPoperator_pol_der} we show that also the corresponding derivative of the resulting policy $\pi_\sigma$ is zero at zero.
	In Lemma~\ref{lem:DPoperator_Valfunc_der} and \ref{lem:DPoperator_state_action_valfunc_der} we derive further results on the partial derivatives of the value functions.

	\begin{lemma}\label{lem:DPoperator_preservation}
	Let Assumptions~\ref{ass:smoothfuncs} to  \ref{ass:regularity_finhorz} hold.
	Let $\bar V_\sigma$ be a smooth value function for which  $\dpartial{}{\sigma} \bar V_\sigma(x)\vert_{\sigma=0} = 0$ for all $x$.
	Then this property is preserved by the DP operator $T_\sigma$, i.e.,
	also for the updated value function $V_\sigma = T_\sigma[\bar V_\sigma]$ it holds that
	$	
	% \begin{equation} \label{eq:valfunc_firstder_zero}
			\dpartialtfrac{}{\sigma} V_\sigma(x)\vert_{\sigma=0} = 0.
		% \end{equation}
	$
	\end{lemma}
	\begin{proof}
		The updated value function is given by
			$V_\sigma(x) = Q_\sigma(x,\pi_\sigma(x))$,
		where $Q_\sigma = T_\sigma^{V\to Q}[\bar V]$, $\pi_\sigma = T_\sigma^{V\to \pi}[\bar V]$.
		The derivative is
	\begin{subequations}  \label{eq:valfunc_first_der}
		\begin{align}
			&\dpartialtfrac{}{\sigma }V_\sigma(x) =
			\dtotaltfrac{}{\sigma} Q_\sigma(x, \pi_\sigma(x)) \\&
			=
			\dpartialtfrac{}{\sigma } Q_\sigma(x, \pi_\sigma(x)) + \dpartialtfrac{}{\sigma } \pi_\sigma(x)^\top \nabla_u Q_\sigma(x, \pi_\sigma(x)) \\&
			= \dpartialtfrac{}{\sigma } Q_\sigma(x, \pi_\sigma(x)),
		\end{align}
	\end{subequations}
	where $\nabla_u Q_\sigma(x, \pi_\sigma(x)) = 0$ due to the definition of $\pi_\sigma$.
		Further,
		\begin{subequations}
			\begin{align}
				&\dpartialtfrac{}{\sigma }Q_\sigma(x,u) =
				\dtotaltfrac{}{\sigma}\EXV_w \{ \bar V_\sigma(f(x,u,\sigma w) ) \} \\
				\label{eq:first_der_Qfunc}
				&\; = \EXV_w \{ \dpartialtfrac{}{\sigma} \bar V_\sigma(f(x,u,\sigma w)  \\ \nonumber
				&\qquad\qquad+ w^\top \nabla_w f(x,u,\sigma w) \nabla_x \bar V_\sigma(f(x,u,\sigma w) \},
			\end{align}
		\end{subequations}
		where the derivative and expectation operator can be swapped due to the measure theoretic Leibniz integral rule \cite[Thm. 2.27]{Folland1999}, which applies due to the bounded noise support and smoothness of $\bar V_\sigma$.
		When evaluating at $\sigma=0$, the first term in \eqref{eq:first_der_Qfunc} drops due to the assumption on $\bar V_\sigma$, while the second term drops due to the zero mean of $w$.
		Then
		$
		% \begin{equation} \label{eq:valfunc_first_der_zero}
			\dpartialtfrac{}{\sigma }V_\sigma(x) \Big\vert_{\sigma=0}
			 = \dpartialtfrac{}{\sigma }Q_\sigma(x,u) \Big\vert_{\sigma=0} = 0
		% \end{equation}
		$
		follows.
		\end{proof}

		\begin{lemma}\label{lem:DPoperator_pol_der}
			Let Assumptions~\ref{ass:smoothfuncs} to  \ref{ass:regularity_finhorz} hold.
			Let $\bar V_\sigma$ be a smooth value function for which  $\dpartial{}{\sigma} \bar V_\sigma(x)\vert_{\sigma=0} = 0$ for all $x$.
			Denote by $\pi_\sigma = T^{V\to\pi}[\bar V_\sigma]$ the  policy associated with the DP operator applied to  $\bar V_\sigma$.
			Its first-order derivative with respect to $\sigma$ is zero at zero,
				% \begin{equation} \label{eq:deriv_greedy_pol_zero}
				$
				\dpartialtfrac{}{\sigma} \pi_\sigma(x) \vert_{\sigma = 0}= 0.
				$
				% \end{equation}
		\end{lemma}
	\begin{proof}
		As before, $\pi_\sigma$ is implicitly defined from $\nabla_u Q_\sigma(x, \pi_\sigma(x)) = 0$.
		Via the implicit function theorem, its derivative is given by 
		\begin{equation*} %\label{eq:deriv_greedy_pol}
			\dpartialtfrac{}{\sigma } \pi_\sigma(x) = - (\nabla_u^2 Q_\sigma(x, \pi_\sigma(x)))^{-1} \dpartialtfrac{}{\sigma} \nabla_u Q_\sigma(x, \pi_\sigma(x)).
		\end{equation*}
		From \eqref{eq:first_der_Qfunc} we have that $\dpartialtfrac{}{\sigma }Q_\sigma(x,u) \vert_{\sigma=0} = 0$ for all $x$, $u$.
		Thus, also
		% \begin{equation}  \label{eq:Qfunc_nab_u_partial_sig}
			$
			\nabla_u \dpartialtfrac{}{\sigma } Q_\sigma(x,u) \vert_{\sigma=0} =
			\dpartialtfrac{}{\sigma } \nabla_u Q_\sigma(x,u) \vert_{\sigma=0} = 0
		% \end{equation}
		$
		holds for all $x$, $u$.
		Therefore, evaluating the policy derivative at $\sigma = 0$ yields 
		$\dpartialtfrac{}{\sigma} \pi_\sigma(x) \vert_{\sigma = 0}= 0$.
	\end{proof}
	\begin{lemma}\label{lem:DPoperator_Valfunc_der}
		Let Assumptions~\ref{ass:smoothfuncs} to  \ref{ass:regularity_finhorz} hold.
		Let $\bar V_\sigma$ be a smooth value function for which  $\dpartial{}{\sigma} \bar V_\sigma(x)\vert_{\sigma=0} = 0$ for all $x$.
		Denote by $V_\sigma = T_\sigma[\bar V_\sigma]$ the updated value function.
		At $\sigma=0$, the derivatives of $V_\sigma$ up to third-order are given by
		% the respective partial derivatives of the corresponding state-action value function $Q=T_\sigma^{V\to Q}[\bar V_\sigma]$, evaluated at corresponding policy  $\pi_\sigma = T^{V\to\pi}[\bar V_\sigma]$, i.e.,
		% for $i=1,2,3$,
		% it holds that
			\begin{equation} \label{eq:val_func_deriv}
				\dpartialtfrac{^i}{\sigma^i }V_{\sigma}(x)\Big\vert_{\sigma=0} = \dpartialtfrac{^i}{\sigma^i }Q_{\sigma}(x, \pi_{\sigma}(x))\Big\vert_{\sigma=0},
			\end{equation}
		for $i=1,2,3$ and with $Q=T_\sigma^{V\to Q}[\bar V_\sigma]$, $\pi_\sigma = T^{V\to\pi}[\bar V_\sigma]$.
	\end{lemma}
	\begin{proof}
		The first-order derivative, i.e., \eqref{eq:val_func_deriv} for $i=1$, is given in \eqref{eq:valfunc_first_der}.
		Continuing this derivation, we get
			\begin{align}
				&\dpartialtfrac{^2}{\sigma^2}V_\sigma(x) =
				\dtotaltfrac{}{\sigma} \dpartialtfrac{}{\sigma } Q_\sigma(x, \pi_\sigma(x))
				\\ \nonumber
				&\;= \dpartialtfrac{^2}{\sigma^2} Q_\sigma(x, \pi_\sigma(x)) +  \dpartialtfrac{}{\sigma } \pi_\sigma(x)^\top \dpartialtfrac{}{\sigma } \nabla_u Q_\sigma(x, \pi_\sigma(x)).
			\end{align}
	
	Evaluating at $\sigma=0$ and noting that the first derivative of $\pi_\sigma(x)$ is zero
	% $\dpartial{}{\sigma } \pi_\sigma(x)\vert_{\sigma=0} = 0$
	from Lemma~\ref{lem:DPoperator_pol_der} yields \eqref{eq:val_func_deriv} for $i=2$.
		The third-order derivative is
		\begin{subequations}
			\begin{alignat}{2}
				&\dpartialtfrac{^3}{\sigma^3}V_\sigma(x)
				\\ \nonumber
				&=
				\dtotaltfrac{}{\sigma} \left(\dpartialtfrac{^2}{\sigma^2} Q_\sigma(x, \pi_\sigma(x)) +  \dpartialtfrac{}{\sigma } \pi_\sigma(x)^\top \dpartialtfrac{}{\sigma } \nabla_u Q_\sigma(x, \pi_\sigma(x))\right)
				\\ \nonumber
				&= \dpartialtfrac{^3}{\sigma^3} Q_\sigma(x, \pi_\sigma(x)) + 
				2 \dpartialtfrac{}{\sigma } \pi_\sigma(x)^\top \dpartialtfrac{^2}{\sigma^2} \nabla_u Q_\sigma(x, \pi_\sigma(x))
				\\ 
				&\qquad +
				\dpartialtfrac{^2}{\sigma^2} \pi_\sigma(x)^\top \dpartialtfrac{}{\sigma }\nabla_u Q_\sigma(x, \pi_\sigma(x))
				\\ \nonumber
				&\qquad
				+
				\dpartialtfrac{}{\sigma } \pi_\sigma(x)^\top
				\nabla^2_u\dpartialtfrac{}{\sigma } Q_\sigma(x, \pi_\sigma(x)) \dpartialtfrac{}{\sigma } \pi_\sigma(x).
			\end{alignat}
		\end{subequations}
		When evaluating at zero, we again have  $\dpartialtfrac{}{\sigma } \pi_\sigma(x)\vert_{\sigma=0} = 0$ from Lemma~\ref{lem:DPoperator_pol_der}. Further, $\dpartialtfrac{}{\sigma } \nabla_u Q_\sigma(x,u) \vert_{\sigma=0} = 0$, cf. the proof of Lemma~\ref{lem:DPoperator_pol_der}.
		Thus only the first term remains, and \eqref{eq:val_func_deriv} for $i=3$ follows, concluding the proof.
	\end{proof}

	\begin{lemma}\label{lem:DPoperator_state_action_valfunc_der}
		Let Assumptions~\ref{ass:smoothfuncs} and  \ref{ass:noise} hold.
		Let $\bar V_\sigma$ be a smooth value function for which  $\dpartial{}{\sigma} \bar V_\sigma(x)\vert_{\sigma=0} = 0$ for all $x$.
		Denote by $Q_\sigma = T^{V\to Q}_\sigma[\bar V_\sigma]$ the resulting state-action value function, and by $\tilde V_\sigma(w; x, u)\defeq \bar V_\sigma(f(x, u, w))$ the corresponding stochastic cost-to-go.
		Then, for $i=2,3,$
		\begin{multline} \label{eq:state_action_val_func_deriv}
			\dpartialtfrac{^i}{\sigma^i }Q_{\sigma}(x, u)\big\vert_{\sigma=0}
			= \dpartialtfrac{^i}{\sigma^i} \bar V_\sigma(f(x,u,\sigma w)) \big\vert_{\sigma=0} 
			\\
			+ \EXV_w \{  \dpartialtfrac{^i}{w^i} \tilde V_\sigma(\sigma w; x,u) \big\vert_{\sigma=0} \bullet w^i \} .
		\end{multline}
	\end{lemma}
	\begin{proof}
	We have 
	% \begin{equation}
		$
		Q_\sigma(x,u) = L(x,u) + \EXV_w \{ \tilde V_\sigma(\sigma w; x, u)  \}
		$
	% \end{equation}
	such that 
	$
	% \begin{equation}
		\dpartialtfrac{^i}{\sigma^i} Q_{\sigma}(x, u) = \EXV_w \{  \dtotaltfrac{^i}{\sigma^i} \tilde V_\sigma(\sigma w; x, u) \},
	% \end{equation}
	$
	for $i=1,2,3$.
	Dropping the dependence on $x,u$ for ease of notation, the total derivatives are
	\begin{subequations}
		\begin{align}
			\dtotaltfrac{^2}{\sigma^2} \tilde V_\sigma(\sigma w)
			&= 
			\dpartialtfrac{^2}{\sigma^2} \tilde V_\sigma(\sigma w)
			+ 2 \dpartialtfrac{}{\sigma} \nabla \tilde V_\sigma(\sigma w)^\top w
			\\&\qquad\qquad\qquad\qquad \nonumber
			+ w ^\top \nabla^2 \tilde V_\sigma(\sigma w) w,
			\\
			\dtotaltfrac{^3}{\sigma^3} \tilde V_\sigma(\sigma w)
			&= 
			\dpartialtfrac{^3}{\sigma^3} \tilde V_\sigma(\sigma w)
			+ 3 \dpartialtfrac{^2}{\sigma^2} \nabla \tilde V_\sigma(\sigma w)^\top w
			\\&\quad \nonumber
			+ 3w^\top \dpartialtfrac{}{\sigma} \nabla^2 \tilde V_\sigma(\sigma w) w
			+ \dpartialtfrac{^3}{w^3} \tilde V_\sigma(\sigma w) \bullet w^3.
		\end{align}
	\end{subequations}
	For evaluation at $\sigma=0$ we first note that  $\dpartial{}{\sigma}\tilde V_\sigma(w)\vert_{\sigma=0} = 0$ for all $w$ by assumption on $\dpartial{}{\sigma} \bar V$.
	In consequence, also $\dpartialtfrac{}{\sigma} \nabla^2 \tilde V_\sigma( w)\vert_{\sigma=0} = 0$.
	When taking the expectation with respect to $w$, all terms that are linear in $w$ after evaluation at $\sigma = 0$ drop due to the zero mean assumption.
	\end{proof}

Based on the previous results, we can now show that the optimal state value function and the state value functions resulting from the evaluation of the CEC policy only differ as $\bigO(\sigma^4)$.

\begin{lemma}\label{lem:taylor_expansions_finhorz}
Let Assumptions~\ref{ass:smoothfuncs} to \ref{ass:regularity_finhorz} hold.
The Taylor expansions in $\sigma$ at $\sigma=0$ of the optimal value function $V^\star_{\sigma,k}(x)$ and the CEC value function $V_{\sigma,k}^{\ce}(x)$, are identical up to including third order, such that, for $k=0,\dots,N-1$,
\begin{equation} \label{eq:Vec_is_Vstar_plus_O4}
	V_{\sigma,k}^\ce(x) = V_{\sigma,k}^{\star}(x) + \bigO(\sigma^4).
\end{equation}
\end{lemma}
\begin{proof}
We will show that from zeroth up to third order, the derivatives with respect to $\sigma$ are identical at $\sigma=0$.
For the zeroth-order term, $V^\ce_{\sigma,k}(x)\vert_{\sigma=0} = V^\star_{\sigma,k}(x)\vert_{\sigma=0}$ follows directly from the definition of the certainty-equivalent policy.

For the first-order derivative we first note that the derivative of the terminal cost with respect to $\sigma$ is trivially given by zero since it does not depend on $\sigma$.
% $\dpartial{}{\sigma} E(x) = 0$.
By Lemma~\ref{lem:DPoperator_preservation}, the DP operator preserves this property, such that $\dpartialtfrac{}{\sigma} V^\ce_{\sigma,k}(x)\vert_{\sigma=0} = \dpartialtfrac{}{\sigma} V^\star_{\sigma,k}(x)\vert_{\sigma=0} = 0$, for $k=N-1,\dots,0$.
In consequence, Lemmata~\ref{lem:DPoperator_pol_der} to \ref{lem:DPoperator_state_action_valfunc_der} apply to each of the value functions.

The second- and third-order derivatives of $V_{\sigma,k}^{\ce}(x) = Q_{\sigma,k}^{\ce}(x, \pi^\ce_k(x))$ are given by the partial derivatives of $Q_{\sigma,k}^{\ce}$, since $\pi^\ce_k(x)$ does not depend on $\sigma$.
From Lemma~\ref{lem:DPoperator_Valfunc_der} we have that the second- and third-order derivatives of $V^\star_{\sigma,k}(x)$ are also given from the partial derivatives of the corresponding state-action value function $Q_{\sigma,k}^\star(x, u)$.

These partial derivatives are given from \eqref{eq:state_action_val_func_deriv} with $\bar V_\sigma(x) = V_{\sigma,k+1}^{\ce}(x)$ resp. $\bar V_\sigma(x) =V^\star_{\sigma,k+1}(x)$, and correspondingly defined $\tilde V_{\sigma,k+1}^{\ce}(w)$, $\tilde V_{\sigma,k+1}^\star(w)$,  for $k=N-1,\dots,0$.
From $V^\ce_{0,k+1}(x) = V^\star_{0,k+1}(x)$ 
follows 
$\tilde V^\ce_{0,k+1}(w) = \tilde V^\star_{0,k+1}(w)$ 
such that also
$\dpartialtfrac{^i}{w^i} \tilde V^\ce_{\sigma,k+1}(w)\vert_{\sigma=0} = \dpartialtfrac{^i}{w^i}  \tilde V^\star_{\sigma,k+1}(w)\vert_{\sigma=0}$ for $i=1,2,3.$
This establishes identity of the second term in \eqref{eq:state_action_val_func_deriv}, when comparing the derivatives of $Q_{\sigma,k}^\ce(x, u)$ and $Q_{\sigma,k}^\star(x, u)$.
The first term in \eqref{eq:state_action_val_func_deriv} is identical if 
\begin{equation} \label{eq:identity_dv_dsig_star_cec}
	\dpartialtfrac{^i}{\sigma^i}  V_{\sigma,k+1}^{\ce}(x) \big\vert_{\sigma=0} = \dpartialtfrac{^i}{\sigma^i}  V_{\sigma,k+1}^{\star}(x) \big\vert_{\sigma=0}, \quad i=1,2.
\end{equation}
For $k=N-1$ this trivially holds due to $V_{\sigma,N}^{\ce}(x) = V_{\sigma,N}^{\star}(x) = E(x)$.
In consequence, for $k=N-1$,
\begin{equation} \label{eq:identity_dv_dsig_star_cec_2}
	\begin{multlined}
	% \begin{split}
			\dpartialtfrac{^i}{\sigma^i}  V_{\sigma,k}^{\ce}(x) \big\vert_{\sigma=0}
			=
			\dpartialtfrac{^i}{\sigma^i}  Q_{\sigma,k}^{\ce}(x, \pi^\ce_{k}(x)) \big\vert_{\sigma=0}
			\\
			=
			\dpartialtfrac{^i}{\sigma^i}  Q_{\sigma,k}^{\star}(x, \pi^\star_{\sigma, k}(x)) \big\vert_{\sigma=0}
			=
			\dpartialtfrac{^i}{\sigma^i}  V_{\sigma,k}^{\star}(x) \big\vert_{\sigma=0},
	% \end{split}
	\end{multlined}
\end{equation}
for $i=2,3.$
This establishes \eqref{eq:identity_dv_dsig_star_cec} also for $k=N-2$.
Repeating this reasoning throughout the DP recursion, we see that \eqref{eq:identity_dv_dsig_star_cec_2} holds also for $k=N-2, \dots, 0$.
Having established that the derivatives with respect to $\sigma$ are identical at $\sigma=0$ up to third order, \eqref{eq:Vec_is_Vstar_plus_O4} immediately follows.
\end{proof}

\begin{theorem} \label{thm:suboptimality_finhorz}
	Let Assumptions~\ref{ass:smoothfuncs} to \ref{ass:regularity_finhorz} hold.
	The suboptimality \eqref{eq:subopt_finhorz} of applying the nominally optimal policy $\poltraj^\ce$ to the stochastic process with uncertainty level $\sigma$ is of fourth order with respect to the level of uncertainty,
	$\Delta V_\sigma(x) = \bigO(\sigma^4)$.
	The difference between the controls is of second order, 
	$\lVert \pi^\ce(x) - \pi_\sigma^\star(x) \rVert = \bigO(\sigma^2)$.
\end{theorem}
\begin{proof}
	% We first note that the suboptimality is always nonnegative, $\Delta V_\sigma(x) \geq 0$ for all $x$, which follows directly from the definition of optimality.
	From Lemma~\ref{lem:taylor_expansions_finhorz} we immediately have $\Delta V_\sigma(x) = V_{0,\sigma}^{\ce}(x) - V_{0,\sigma}^\star(x) = \bigO(\sigma^4)$.
	The statement on the controls follows from a second-order Taylor expansion and noting that
	$\dpartialtfrac{}{\sigma } \pi^\star_\sigma(x) \vert_{\sigma=0} = 0$ (Lemma~\ref{lem:DPoperator_pol_der}) and $\dpartialtfrac{}{\sigma} \pi^\ce(x) = 0$ (trivially).
	\end{proof}

\section{Numerical Illustration} \label{sec:num}

We will now illustrate the results with a simple example, which allows us to evaluate the value functions and policies up to numerical precision.
This example is implemented via the Python interface of CasADi \cite{Andersson2019} with IPOPT \cite{Waechter2006} as solver.
The code is publicly available at \url{www.github.com/fmesserer/suboptimality-nominal-mpc}.

Consider the scalar state $x \in \R$, control $u\in \R$ and disturbance $w\in\R$.
The continuous time dynamics, over the time interval $[0,T]$, are given by
% \begin{equation}
	$
	\dot x = x + x^3 + u
	$
% \end{equation}
from which we obtain the discrete time dynamics 
\begin{equation} \label{eq:example_dyn}
	x_{k+1} = f_\mathrm{RK4}(x_k, u_k) +  \sigma w_k, \quad k=0,\dots,N-1,
\end{equation}
by numerical integration with one step of the Runge-Kutta method of fourth order, where the controls $u_k$ are piecewise constant over the time step $h=T/N$.
The disturbance $w_k$ follows a discrete distribution and takes values from the set $W = \{-1, 1\}$ with probability $p=\tfrac{1}{2}$ for each value.

\begin{figure}
	\vspace{5pt}
	\centering
	\includegraphics[width=.9\columnwidth]{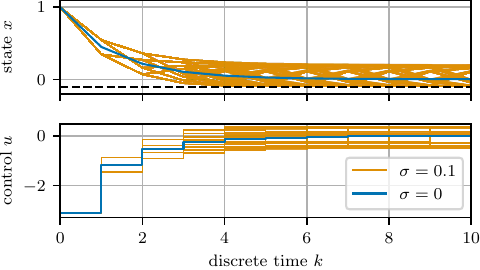}
	\caption{Two solutions of the tree OCP \eqref{eq:treeocp} with different levels of uncertainty for the initial value $x=1$.}
	\label{fig:solution_treeocp}
\end{figure}

The control goal is to stabilize the system near the origin while keeping the state above a lower bound, $x \geq x_\mathrm{lb}$.
This is expressed by the stage cost
\begin{equation} \label{eq:example_stagecost}
	L(x, u) = q x^2 + r u ^2 + \rho \phi_\varepsilon(x - x_\mathrm{lb}),
\end{equation}
which is visualized in Fig.~\ref{fig:stage_cost}.
Here, the first two terms are standard quadratic costs. The third term, with
% \begin{equation}
	$
	\phi_\varepsilon(x) \defeq \tfrac{1}{2}\sqrt{x^2 +  \varepsilon^2 } - \tfrac{1}{2} x,
	$
% \end{equation}
is a smoothed overapproximation of the exact penalty \sloppy $\max(0, -x)$, with smoothing parameter $\varepsilon$ and penalty weight $\rho$.
The terminal cost is 
% \begin{equation}
	$
	E(x) = \tilde q x^2 + \rho \phi_\varepsilon(x-x_\mathrm{lb}),
	$
% \end{equation}
where $\tilde q$ is chosen as the solution to the
% discrete time
algebraic Riccati equation corresponding to the infinite horizon LQR problem obtained by linearizing the system at the origin and with cost matrices $q$ and $r$.
The parameter values are $T=2$, $N=10$, $x_\mathrm{lb}=-0.1$, $q=5$, $r=1$, $\rho = 10$, $\varepsilon = 10^{-2}$.

Due to the discrete distribution of the process noise, the set of possible trajectories can be described by a scenario tree.
Starting from the root $\bar x_0$, the number of scenarios is multiplied by $m=\lvert \calW\rvert$ with every time step.
% , i.e., one new scenario for each value the disturbance may take.
We denote the possible values of the state at time $k$ by $x_k^i$, $i=1,\dots,m^k$, $k=0,\dots,N$.
Associating a control $u_k^i$ with every node implicitly encodes a policy, as the control value depends on the realized state.
The resulting stochastic OCP corresponds to \eqref{eq:stochOCP} and takes the form of a tree-structured OCP,
\begin{mini!}
	{\xtree, \utree}
	{ \sum_{k=0}^{N-1}  \left( p^k  \sum_{i=1}^{m^k} L(x_k^i, u_k^i) \right) + p^N \sum_{i=1}^{m^N} E(x_N^i)      }
	{\label{eq:treeocp}}
	{}
	\addConstraint{x_0^0}{= x}
	\addConstraint{x_{k+1}^i }{= f(x_k^{\lceil i / m^k \rceil}, u_k^{\lceil i / m^k \rceil}, \sigma w^{\left.i\right]_1^m}) \label{eq:tree_dyn} }%{i=1,\dots, m^{k+1}}
	\addConstraint{}{\qquad i=1,\dots, m^{k+1}, k=0,\dots, N-1, \nonumber}
\end{mini!}
where $\lceil \cdot \rceil$ denotes the ceiling function and $\left.i\right]_1^m$
% \defeq \modmod(i + 1,m) + 1$
wraps the integer~$i$ to the set $\{1, \dots, m\}$.
Thus, for each $k=0,\dots,N-1$, the dynamics constraint \eqref{eq:tree_dyn} cycles through all scenarios of the current stage, $(x_k^i, u_k^i)$, $i=1,\dots,m^k$, and simulates it forward once for every possible disturbance value, $w^i \in \calW$.
The resulting scenario trees are collected in $\xtree$ resp. $\utree$. 
For the cost contribution, the tree nodes are summed up within each stage and weighted by their respective probability.
% , which here is identical within each stage and given by $p^k$ at stage~$k$.
Two solutions of \eqref{eq:treeocp} are visualized in Fig.~\ref{fig:solution_treeocp}.
The nominal OCP with $\sigma = 0$ can equivalently be written as the degenerate tree OCP resulting from the singleton disturbance set $\calW = \{0\}$.

We compute the suboptimality of CEC for varying values of $x$ and $\sigma$.
The optimal value function $V^\star_\sigma(x)$ is given by the optimal value of \eqref{eq:treeocp}.
The value function of CEC results from simulating the system \eqref{eq:example_dyn} for every possible value of $(w_0,\dots,w_{N-1})$. At each time $k$ and for each possible state $x_k^i$, the corresponding control $u_k^i$ is computed by solving the nominal OCP over a reduced horizon $\tilde N = N - k$.
The costs are then summed up as in the objective of \eqref{eq:treeocp}.

Fig.~\ref{fig:stage_cost} shows the value functions for several values of $\sigma$, whereas Fig.~\ref{fig:subopt_and_delta_u0} visualizes the suboptimality and difference in policy as a function of $\sigma$ for several values of $x$.
We see that for small values of $\sigma$ the suboptimality is $\bigO(\sigma^4)$ and the difference in policy is $\bigO(\sigma^2)$, as predicted by the theory.
At roughly $\sigma \approx 0.1$, this relationship ceases to hold for  larger values of $x$.
Even though we have only enforced the bound $x\geq x_\mathrm{lb}$ via a penalty, we can think about this event roughly as this constraint becoming active.
With respect to the stage cost, this corresponds to the strongly nonlinear region around $x\approx x_\mathrm{lb}$ influencing the solution, cf. Fig.~\ref{fig:stage_cost}.
Strong nonlinearity here means that the higher-order terms of the Taylor expansions are significant and start to influence the solution already at relatively small uncertainty levels.

\begin{figure}
	\vspace{5pt}
	\centering
	\includegraphics[width=\columnwidth]{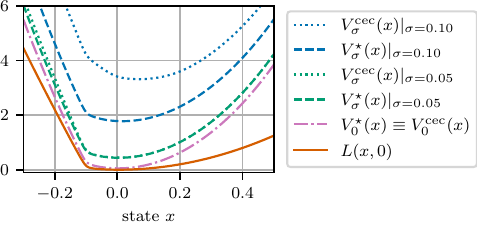}
	\caption{The value function $V_\sigma^\ce$ of CEC and the optimal value function $V_\sigma^\star$ for three values of $\sigma$,  as well as the stage cost $L$ with respect to the state~$x$.
	For $\sigma=0.05$, the two value functions are barely distinguishable.
	}
	\label{fig:stage_cost}
\end{figure}

\begin{figure}
	\vspace{5pt}
	\centering
	\includegraphics[width=\columnwidth]{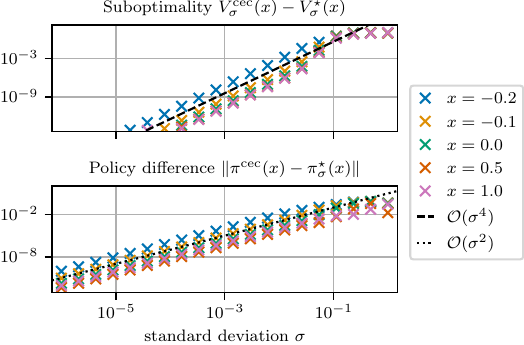}
	\caption{
		Suboptimality of CEC and difference of the control inputs as a function of $\sigma$  and for several values of $x$.
		For $x=0.5$ and $x=1.0$ the suboptimality is almost identical.
	}
	\label{fig:subopt_and_delta_u0}
\end{figure}

\section{Conclusions} \label{sec:conc}
We have shown that in a smooth and unconstrained setting the suboptimality of CEC grows only with fourth order as the level of uncertainty increases.
This suggests that uncertainty-aware MPC schemes are able to significantly outperform nominal MPC only in the presence of large disturbances or constraints.
% This suggests that only in the presence of large disturbances or state constraints uncertainty-aware MPC schemes will be able to significantly outperform nominal MPC 

% This provides a justification for the use of nominal MPC. For comparably small disturbances, we would not expect significant performance gains from uncertainty-aware MPC schemes.
% However, the smoothness assumption most notably excludes constraint and exact, nonsmooth penalties.
% The fourth-order result suggests that uncertainty-aware MPC schemes do only significantly outperform nominal MPC in the presence of large disturbances or state constraints.

% \addtolength{\textheight}{-12cm}
% This command serves to balance the column lengths on the last page of the document manually.
% It shortens the textheight of the last page by a suitable amount.
% This command does not take effect until the next page so it should come on the page before the last.
% Make sure that you do not shorten the textheight too much.

\bibliographystyle{IEEEtran}
\bibliography{syscop}

% \clearpage

% \section{Open points / remarks}
% 	\begin{itemize}
% 		\item 6 pages max, already upon submission!
% 		\item title suggestions?
% 		\begin{itemize}
% 			\item \emph{Fourth-order suboptimality of nominal model predictive control in the presence of uncertainty}
% 			\item \emph{Suboptimality analysis of nominal model predictive control in the presence of uncertainty}
% 			\item \emph{Suboptimality of nominal model predictive control is of fourth order with respect to the level of uncertainty}
% 			\item ...
% 		\end{itemize}
% 	\end{itemize}

% \appendix

\end{document}

%% file: settings.tex
\usepackage{graphics}
\usepackage{amsmath, amssymb, mathtools, amsfonts}
\usepackage{bm}
\usepackage{mathabx}
\usepackage{color}
\usepackage[short]{optidef}
\usepackage{mleftright}

\let\labelindent\relax
\usepackage{enumitem}

\usepackage{amsthm}

\makeatletter
\let\NAT@parse\undefined
\makeatother
\usepackage[hidelinks]{hyperref}

\theoremstyle{plain}
\newtheorem{theorem}{Theorem}%[section]
\newtheorem{lemma}[]{Lemma}

\theoremstyle{definition}

\newtheorem{remark}[]{Remark}
\newtheorem{assumption}[]{Assumption}

\newcommand{\bmat}{\begin{bmatrix}}
\newcommand{\emat}{\end{bmatrix}}
\newcommand{\dpartial}[2]{\frac{\partial{#1}}{\partial{#2}}}
\newcommand{\dpartialtfrac}[2]{\tfrac{\partial{#1}}{\partial{#2}}}

\newcommand{\dtotaltfrac}[2]{\tfrac{\mrm{d} #1}{\mrm{d} #2}}

\newcommand{\argmin}{\arg\min}

\DeclareMathOperator{\EXV}{\mathbb{E}}		% expected value
\newcommand{\R}{\mathbb{R}}

\newcommand{\bigO}{\mathcal{O}}
\newcommand{\calW}{\mathcal{W}}

\newcommand{\defeq}{:=}

\newcommand{\eye}{I}  % identity matrix

\newcommand{\mrm}{\mathrm}

\newcommand{\ce}{\mrm{cec}}

\newcommand{\traj}[1]{ \check{#1} }
\newcommand{\xtraj}{\traj{x}}
\newcommand{\utraj}{\traj{u}}
\newcommand{\wtraj}{\traj{w}}
\newcommand{\poltraj}{\traj{\pi}}
\newcommand{\tree}{\rotatebox{90}{$\scriptstyle\pitchfork$}}
\newcommand{\xtree}{x_{\tree}}
\newcommand{\utree}{u_{\tree}}